\documentclass[a4paper,10pt]{article}
\usepackage[bbgreekl]{mathbbol}
\usepackage[small]{titlesec}
\usepackage[utf8]{inputenc}
\usepackage[T1]{fontenc}  
\usepackage[]{mdframed}
\usepackage{mathrsfs}
\usepackage[backend=biber,style=alphabetic,natbib=false,giveninits=true,doi=true,isbn=false,url=false,date=year,maxbibnames=99,sorting=nty,defernumbers=true]{biblatex}
\DeclareNameAlias{default}{family-given}

\DeclareFieldFormat[article]{title}{#1}
\renewbibmacro{in:}{}
\DeclareFieldFormat[article]{pages}{#1}
\DeclareFieldFormat{journal}{#1}
\renewcommand\bf\bfseries

\renewbibmacro*{volume+number+eid}{%
	\printfield{volume}%
	\setunit*{\addnbspace}
	\printfield{number}%
	\setunit{\addcomma\space}%
	\printfield{eid}}
\DeclareFieldFormat[article]{number}{\mkbibparens{#1}}
\DeclareFieldFormat[article]{volume}{\textbf{#1}}
\DeclareFieldFormat{year}{\mkbibparens{#1}}
\DeclareBibliographyDriver{article}{%
	\printnames{author}:%
	\newunit\newblock
	\printfield{title}%
	\newunit\newblock
	\printfield{journaltitle}
	\newunit
	\iffieldundef{number}{\printfield{volume}}{\printfield{volume}\addspace\printfield{number}}
	\addcomma\addspace\printfield{pages}\addspace
	\printfield{year}
}
\AtEveryBibitem{\clearfield{month}}
\AtEveryBibitem{\clearfield{day}}
\usepackage[english]{babel}
\usepackage[T1]{fontenc}
\usepackage{csquotes}
\usepackage{bbm}
\usepackage[leqno]{amsmath}
\usepackage{amsfonts,amsthm,amsbsy,amssymb,dsfont,stmaryrd}
\usepackage[dvipsnames]{xcolor}
\usepackage{braket}

\usepackage{caption}
\usepackage{subcaption}

\usepackage{enumitem}

\makeatletter
\newcommand{\leqnomode}{\tagsleft@true\let\veqno\@@leqno}
\newcommand{\reqnomode}{\tagsleft@false\let\veqno\@@eqno}
\makeatother

\usepackage{mathtools}
\usepackage[makeroom]{cancel}

\numberwithin{equation}{section}

\newcommand\myshade{85}
\colorlet{mylinkcolor}{violet}
\colorlet{mycitecolor}{YellowOrange}
\colorlet{myurlcolor}{Aquamarine}

\usepackage[left=2.5cm,right=2.5cm,top=2.5cm,bottom=2.5cm]{geometry}
\usepackage[unicode=true,pdfusetitle,
bookmarks=true,bookmarksnumbered=false,bookmarksopen=false,
breaklinks=false,pdfborder={0 0 1},backref=false,
linkcolor  = mylinkcolor!\myshade!black,
citecolor  = mycitecolor!\myshade!black,
urlcolor   = myurlcolor!\myshade!black,
colorlinks = true,
]
{hyperref}
\usepackage[nameinlink]{cleveref}

\usepackage{ifthen}

\usepackage{graphicx}
\usepackage{caption}
\usepackage{slashed}
\definecolor{ct_black}{HTML}{000000}
\definecolor{ct_orange}{HTML}{ED872D}
\definecolor{ct_purple}{HTML}{7A68A6}
\definecolor{ct_blue}{HTML}{348ABD}
\definecolor{ct_turquoise}{HTML}{188487}
\definecolor{ct_red}{HTML}{E32636}
\definecolor{ct_pink}{HTML}{CF4457}
\definecolor{ct_green}{HTML}{467821}

\definecolor{ct2_green}{HTML}{9FF781}
\definecolor{ct2_green_dark}{HTML}{088A08}

\theoremstyle{plain}
\newtheorem{thm}{\protect\theoremname}[section]
\theoremstyle{plain}
\newtheorem{lem}[thm]{\protect\lemmaname}
\theoremstyle{plain}
\newtheorem{cor}[thm]{\protect\corollaryname}
\theoremstyle{plain}
\newtheorem{prop}[thm]{\protect\propositionname}
\theoremstyle{plain}

\theoremstyle{remark}
\newtheorem{rem}[thm]{\protect\remarkname}

\theoremstyle{definition}
\newtheorem{defn}[thm]{\protect\definitionname}
\theoremstyle{plain}

\providecommand{\assumptionname}{Assumption}
\providecommand{\claimname}{Claim}
\providecommand{\corollaryname}{Corollary}
\providecommand{\definitionname}{Definition}
\providecommand{\lemmaname}{Lemma}
\providecommand{\propositionname}{Proposition}
\providecommand{\remarkname}{Remark}
\providecommand{\theoremname}{Theorem}
\providecommand{\examplename}{Example}

\crefname{section}{Section}{Sections}
\crefname{example}{Example}{Examples}
\crefname{appendix}{Appendix}{Appendices}
\crefname{figure}{Figure}{Figures}
\crefname{assumption}{Assumption}{Assumptions}
\crefname{thm}{Theorem}{Theorems}
\crefname{lem}{Lemma}{Lemmas}
\crefformat{equation}{(#2#1#3)}
\crefname{table}{Table}{Tables}

\crefrangelabelformat{equation}{(#3#1#4--#5#2#6)}

\crefmultiformat{equation}{(#2#1#3}{, #2#1#3)}{#2#1#3}{#2#1#3}
\Crefmultiformat{equation}{(#2#1#3}{, #2#1#3)}{#2#1#3}{#2#1#3}

\newtheorem*{lem*}{\protect\lemmaname}

\newcommand{\ii}{\operatorname{i}}

\newcommand{\ZZ}{\mathbb{Z}}

\newcommand{\bS}{\mathbb{S}}

\newcommand{\NN}{\mathbb{N}}
\newcommand{\RR}{\mathbb{R}}
\newcommand{\QQ}{\mathbb{Q}}
\newcommand{\CC}{\mathbb{C}}

\newcommand{\PP}{\mathbb{P}}

\newcommand{\calA}{\mathcal{A}}
\newcommand{\calB}{\mathcal{B}}
\newcommand{\calC}{\mathcal{C}}

\newcommand{\calN}{\mathcal{N}}
\newcommand{\calG}{\mathcal{G}}

\newcommand{\calD}{\mathcal{D}}

\newcommand{\calU}{\mathcal{U}}

\newcommand{\calH}{\mathcal{H}}

\newcommand{\calK}{\mathcal{K}}
\newcommand{\calL}{\mathcal{L}}

\newcommand{\calP}{\mathcal{P}}
\newcommand{\calS}{\mathcal{S}}

\newcommand{\ti}[1]{\widetilde{#1}}

\newcommand\norm[1]{\left\lVert#1\right\rVert}
\newcommand{\ip}[2]{\langle #1, #2 \rangle}

\newcommand{\ve}{\varepsilon}
\newcommand{\vf}{\varphi}
\newcommand{\Id}{\mathds{1}}

\newcommand{\findex}{\operatorname{ind}}

\newcommand{\supp}{\operatorname{supp}}

\newcommand{\im}{\operatorname{im}}

\usepackage{environ}

\NewEnviron{malign}{%
	\begin{align}\begin{split}
			\BODY
	\end{split}\end{align}
}

\newcommand{\eq}[1]{\begin{align*}#1\end{align*}}
\newcommand{\eql}[1]{\begin{align}#1\end{align}}

\newcommand{\NT}{\mathrm{nt}}
\newcommand{\calQ}{\mathcal{Q}}
\newcommand{\br}[1]{\left(#1\right)}

\newcommand{\Ext}{\operatorname{Ext}}

\title{Essentially Commuting with a Unitary}
\author{\href{mailto:jc1220@math.princeton.edu}{Jui-Hui Chung}\\
	{\footnotesize Program in Applied and Computational Mathematics, Princeton University }\\
	\href{mailto:jacobshapiro@princeton.edu}{Jacob Shapiro}\\
	{\footnotesize Department of Mathematics, Princeton University}
}

\begin{document}
\reqnomode

\maketitle
\begin{abstract}
Let $R$ be a unitary operator whose spectrum is the circle. We show that the set of unitaries $U$ which essentially commute with $R$ (i.e., $[U,R]\equiv UR-RU$ is compact) is path-connected. Moreover, we also calculate the set of path-connected components of the orthogonal projections which essentially commute with $R$ and obey a non-triviality condition, and prove it is bijective with $\ZZ$.
\linebreak
\end{abstract}
\tableofcontents

\section{Introduction and the result}

Let $\calH$ be a separable Hilbert space, $\calB(\calH)$ the space of bounded operators on $\calH$, and $\calK(\calH)$ the space of compact operators on $\calH$. Let $R$ be a unitary operator on $\calH$ such that $\sigma(R)=\bS^1$. In this paper we are interested in operators that essentially commute with $R$, in particular in the space of unitaries and orthogonal projections which obey this constraint. In anticipation of application in mathematical physics, we use the terminology that the operator $A\in\calB(\calH)$  is ``$R$-local'' iff
\eql{
    [A,R]:=AR-RA\in\calK(\calH)\,.
} Clearly the subset of $\calB(\calH)$ of operators which essentially commute with $R$ is a $C^\ast$-algebra, we denote it by $\calL_R$ and call it the ``$R$-local algebra'', its norm is induced by the operator norm on $\calB(\calH)$.

We refer the reader to \cite[Section 8]{ChungShapiro2023} for  motivation in the  mathematical physics of topological insulators for studying the space of unitaries $\calU(\calL_R)$ and self-adjoint projections $\calP(\calL_R)$ within $\calL_R$. In one sentence: $R$ would be the Dirac phase, and so we are classifying either class A (projections) or class AIII systems (unitaries) respectively in even space dimensions. Non-triviality has to do with systems which are honestly bulk systems. The details of this application are postponed to a forthcoming paper. 

From the point of view of pure mathematics, operators which essentially commute with a fixed non-trivial projection $\Lambda$ were studied in \cite{CareyHurstOBrien1982,AndruchowChiumientoLucero2015} for unitaries and orthogonal projections respectively, so it is natural to generalize the question from a fixed non-trivial projection (having $\Set{0,1}$ in its essential spectrum) $\Lambda$ to a fixed unitary $R$ with $\sigma(R)=\bS^1$. The conclusion of \cite{CareyHurstOBrien1982} is that unitaries which essentially commute with $\Lambda$ have path-connected components which are bijective with $\ZZ$. The orthogonal projections which essentially commute with $\Lambda$ have a more complicated structure \cite{AndruchowChiumientoLucero2015}, but if one further restricts to the set of projections $P$ such that both $\sigma_{\mathrm{ess}}(\Lambda P \Lambda)$ and $\sigma_{\mathrm{ess}}(\Lambda^\perp P \Lambda^\perp)$ are exactly $\Set{0,1}$, then that subset of projections is path-connected. Our present $R$-non-triviality condition is a generalization of this idea to the unitary case.

Our first main theorem is
\begin{thm}[Classification of $R$-local unitaries]\label{thm:R-local unitaries are path-connected} The space$\,$ $\calU(\calL_R)$ of $R$-local unitaries is path-connected:
\eql{
    \pi_0(\calU(\calL_R)) \cong \Set{0}\,.
}
\end{thm}
Here and in the sequel we shall always take the topology to be the subspace topology from the operator norm topology on $\calB(\calH)$.

Next we turn to the space of orthogonal projections essentially commuting with $R$. Unfortunately we do not have a statement about this entire space, but rather, only a subset of it which we term \emph{R}-non-trivial.
\begin{defn}\label{defn:non-triviality}
    An orthogonal projection $P\in\calP(\calL_R)$ is termed $R$-non-trivial iff for any continuous non-zero $f:\bS^1\to\CC$, both $P f(R) P$ and $P^\perp f(R) P^\perp$ are \emph{not} compact operators. 

    We denote the space of all $R$-non-trivial orthogonal projections as $\calP^{\mathrm{nt}}_R$ and furnish it with the subspace topology from $\calB(\calH)$. 
\end{defn}

\begin{rem}
In the language of $K$-theory of dual $C^\ast$-algebras, the orthogonal projection $P$ is $R$-non-trivial iff $P$ is \emph{ample} with respect to the dual $C^\ast$-algebra \cref{eq:dual of C(S) in R rep}. See \cite[Definition 5.1.3]{higson2000analytic}.
\end{rem}

We note that for any $R$-local element $P\in\calP(\calL_R)$, there is a well-defined continuous integer associated to it, given by 
\eql{ \label{eq:index of R-local projections}
    \calP(\calL_R) \ni P \mapsto \findex(\PP R) =: \calN_R(P) \in \ZZ
} where we use the shorthand notation \eql{
    \PP R := P R P + P^\perp\,.
} Indeed, the fact $[P,R]$ is compact implies, by Atkinson, that $\PP R$ is Fredholm with parametrix $\PP R^\ast$.

With this we have the second main
\begin{thm}[Classification of $R$-non-trivial orthogonal projections]\label{thm:classification of R-non-trivial orthogonal projections}
    The continuous map 
    \eql{\label{eq:projection index}
    \calN_R:\calP_R^{\mathrm{nt}}\to\ZZ
    } lifts to a bijection on path-connected components
    \eql{
        \pi_0(\calP_R^{\mathrm{nt}}) \cong \ZZ\,.
    }

    Moreover, $\calP_R^{\mathrm{nt}}$ is a maximal component within $\calP(\calL_R)$, in the sense that if $P\in \calP_R^{\mathrm{nt}}$, $Q\in\calP(\calL_R)$ and $P,Q$ are in the same path-connected component of $\calP(\calL_R)$ then actually $Q\in\calP_R^{\mathrm{nt}}$.
\end{thm}

At the time being we have little to say about the path-components of $\calP(\calL_R)\setminus\calP_R^{\mathrm{nt}}$.

In the process of writing this manuscript, we learnt of an alternative proof of \cref{thm:R-local unitaries are path-connected} which uses the theory of $K_1$-injectivity, see \cref{sec:remark on K-theory and our main results}. In that sesne, one could say that \cref{thm:R-local unitaries are path-connected} is an alternative proof (to that of Paschke \cite{paschke1981k}) of $K_1$-injectivity for the particular algebra of $R$-local operators.

\begin{rem}
    One may ask why the constraint that $R$ has spectrum on the whole unit circle is important. First of all, it is clear by the Weyl-von Nuemann-Berg theorem that only the essential spectrum of $R$ matters. Second, by the analysis in \cite{chung2024topological}, if $\sigma_{\mathrm{ess}}(R)$ only has finitely many points, then the classification is covered by \cite[Theorems 4.1 and 5.5]{chung2024topological} and is \emph{different} than what is presented here for $\sigma(R)=\bS^1$. It is at the moment not clear to us what happens in more general cases of $\sigma_{\mathrm{ess}}(R)$. For our application in mathematical physics, however, only the case presented here is important.
\end{rem}

In the rest of this section we outline our general strategy for the proof of our two theorems and set up some notation. \cref{sec:unitaries classification proof,sec:orthogonal projections classification proof} contain the proofs of our two main theorems, minus the technical details, which are relegated to \cref{sec:details of the proof}. The appendices contain the consequences of locality and boundedness, the $K$-groups calculations of the local algebra and a discussion of $K_1$-injectivity.

\subsection{Strategy for the proof}\label{sec:strategy}
The general strategy of classifying operators that essentially commute with a fixed non-trivial projection $\Lambda$ was relatively straightforward: if $[A,\Lambda]$ is compact, then in the decomposition $\calH = \im \Lambda^\perp \oplus \im \Lambda$ we write \eql{\label{eq:Lambda-decomposition of an operator} A = \begin{bmatrix}
    A_{LL} & A_{LR} \\ A_{RL} & A_{RR}
\end{bmatrix}} and the compact condition is equivalent to the off-diagonal blocks being compact. General algebraic considerations then lead us to figure out the properties of the diagonal blocks that allow deformations.

To follow with this approach in the unitary case would be difficult: one key difference is that $\Lambda$ has only two eigenspaces whereas $R$ has infinitely many, so the analog of \cref{eq:Lambda-decomposition of an operator} is not clear; we explored questions of similar flavor in \cite{chung2024topological}.

Our strategy here is rather in-direct and is divided into the following steps:
\begin{enumerate}
    \item Calculate the $K$-theory of the $C^\ast$-algebra $\calL_R$, which is standard and given in \cref{sec:K-theory of the R-local algebra}. The result is $K_0(\calL_R)\cong \ZZ$ and $K_1(\calL_R)\cong 0$.
    \item Show that equivalence in $K$-theory implies equivalence of path-components. To do so, we choose a particular representation of $R$ and use the geometry of the problem. The geometry of local operators is discussed in \cref{sec:locality and boundedness}.
\end{enumerate}

We emphasize that, for unital $C^\ast$-algebra $\calA$, the fact that the $K$-groups $K_0(\calA)$ and $K_1(\calA)$ agree with $\pi_0(\calP(\calA))$ and $\pi_0(\calU(\calA))$ respectively is highly non-trivial and is in fact false: this is the very reason we had to restrict here to the non-trivial subset within the projections. To that end we give examples of algebras in \cref{sec:counter-examples} where there is a stark difference between $K$-equivalence and homotopy equivalence.


\subsubsection{Working with a concrete representation of $R$}\label{subsubsection:concrete rep}

To employ some geometric intuition, we choose to work with a concrete representation of $R$ (which is in fact the main motivation for studying this problem). In \cite[Prop. 2.1]{chung2024topological} we showed that any choice of representation will yield the same topological classification, as long as the unitary operator chosen has spectrum on the whole unit circle.

Consider then the Hilbert space $\ell^2(\ZZ^2)$ and the so-called Laughlin flux operator $L$ on $\ell^2(\ZZ^2)$ defined as
\eql{
    L = \frac{X_1+\ii X_2}{\left|X_1+\ii X_2\right|} = \exp\br{\ii \arg\br{X_1+\ii X_2}}\,.
} We honor Laughlin with this name due to the fact that this operator was used in writing the index formula corresponding to Laughlin's explanation of the integer quantum Hall effect quantization \cite{Bellissard_1994JMP....35.5373B}. It is clear that $\sigma(L)=\bS^1$.

\subsubsection{Notation}

From now on, we will denote $\calH$ to be the Hilbert space $\ell^2(\ZZ^2)$. Let $\calL_L$ be the space of $L$-local operators. Let $\calP(\calL_L)$ be the set of $L$-local (orthogonal) projections, and $\calU(\calL_L)$ the set of $L$-local unitary operators, and $\calG(\calL_L)$ the set of $L$-local invertible operators.

Let $P,Q\in\calP(\calL_L)$ be projections, we shall use the following equivalence classes on projections:
\begin{enumerate}
    \item Homotopy equivalence: $P\sim_h Q$ iff there is a continuous path of projections in $\calP(\calL_L)$ from $P$ to $Q$.
    \item Unitary equivalence: $P \sim_u Q$ iff there exists a unitary $V\in\calU(\calL_L)$ such that $P = V^\ast Q V$. 
    \item Muarry-von Neumann equivalence: $P\sim Q$ iff there exists an operator $V\in\calL_L$ such that $P=V^\ast V$ and $Q=VV^\ast$.
\end{enumerate} Recall, from \cite[Proposition 2.2.7]{rordam2000introduction}, that the first condition implies the second, which in turn implies the third.

For unitaries, if $U,V\in\calU(\calL_L)$, then we write $U\sim_h V$ iff there exists a continuous path within $\calU(\calL_L)$ from $U$ to $V$.

We denote by $M_n(\calL_L)$ the matrix algebra of $n\times n$ matrices with entries in $\calL_L$, itself a $C^\ast$-algebra. Let $\calP_n(\calL_L):=\calP(M_n(\calL_L))$ and $\calU_n(\calL_L):=\calU(M_n(\calL_L))$. Let $\calP_\infty(\calL_L)$ be the disjoint union of $\calP_n(\calL_L)$ for all $n\in\NN$, and $\calU_\infty(\calL_L)$ be the disjoint union of $\calU_n(\calL_L)$ for all $n\in\NN$.

For $P\in\calP_n(\calL_L)$ and $Q\in\calP_m(\calL_L)$, we write $P\sim_0 Q$ if there exists some partial isometry $V\in M_{m,n}(\calL_L)$ such that $P=V^\ast V$ and $Q=VV^\ast$. Here $M_{m,n}(\calL_L)$ is the set of $m\times n$ matrices with entries in $\calL_L$. We say that two projections $P,Q\in\calP_\infty(\calL_L)$ are stably equivalent, denoted $P\sim_s Q$, iff there exists some $n\in\NN$ such that $P\oplus \Id_n \sim_0 Q\oplus \Id_n$.



\begin{defn}[Projections onto subsets of $\ZZ^2$ lattice]\label{def:proj onto subsets of Z2}
If $S\subset \ZZ^2$ is a subset of lattice points, we denote by $\Lambda_S$ the projection operator \eql{\Lambda_S=\sum_{x\in S}\delta_x \otimes \delta_x^\ast\,.} For convenience, we define the following subsets of $\ZZ^2$:
\begin{enumerate}
\item If $J\subset \bS^1$ is an interval of the circle, we denote $C_J:= \Set{x\in\ZZ^2 | \arg(x)\in J}$. This is the \emph{cone} defined by the arc $J$.
\item Let $r>0$ be a positive number. We denote ${B_r} := \Set{x\in \ZZ^2 | \|x\|< r}$, the open $r$-ball.
\item For $\vf\in\ell^2(\ZZ^2)$, define $\supp(\vf):=\Set{x\in\ZZ^2 | \vf_x\neq 0}$.
\end{enumerate}
If $J\subset\bS^1$ is an interval of the circle, we will denote
\eq{
    \Lambda_{J} := \Lambda_{C_J}\equiv \sum_{x\in C_J}\delta_x\otimes \delta_x^\ast\,.
}
\end{defn}

\begin{defn}[Angles of rational slopes]
We define $\bS^1_\QQ\subset \bS^1$ as the set of points $z\in\CC$ with $|z|=1$ such that $\tan\br{\operatorname{Arg}(z)}\in\QQ$. Let $z_1,z_2\in\bS^1_\QQ$ be two points. We write $[z_1,z_2]\subset\bS^1_\QQ$ to be those points in $\bS^1_\QQ$ swept from $z_1$ to $z_2$ (inclusive) in a counter-clockwise fashion.
\end{defn}

\section{\texorpdfstring{The proof of \cref{thm:R-local unitaries are path-connected}: classification of unitaries }{}}\label{sec:unitaries classification proof}


\cref{thm:R-local unitaries are path-connected} is equivalent to $U\sim_h\Id$ for all $U\in\calU(\calL_L)$. Our general strategy is as follows: We seek a special projection $P$ for which we may construct a homotopy 
\eql{\label{eq:unitary first homotopy}
    U\sim_h W = P + P^\perp WP^\perp
} for some $W\in\calU(\calL_L)$. 
The art will be in choosing $P$ to be ``sufficiently infinite'' as it were. If we can manage this with $P=\Id$, then we are done, but this is too difficult, so we break up this task. A more modest goal is the following. Seek a projection $P$ which is sufficiently infinite in the sense that,
\eql{\label{eq:infintie projections}
P\oplus \Id_n \sim_0 P\qquad(n\in\NN)\,.
}
By the above, this is equivalent to: for every $n\in\NN$ there exists some partial isometry $V$ such that \eql{ P = V^\ast V \,,\qquad P\oplus\Id_n = VV^\ast \,.}
Heuristically, this means we can compress any ``amplification'' of $\im P$ by $n$-copies of $\calH$ back into $\im P$. It turns out that, indeed, with this property, we have
\begin{prop} \label{prop:block unitary homotopy}
If $P\in\calP(\calL_L)$ obeys \cref{eq:infintie projections} and $U\in\calU(\calL_L)$ is of the form 
\eql{\label{eq:U acts as identity within an infinite projection}
U=P+P^\perp UP^\perp
}
then $U \sim_h \Id$.
\end{prop}
The proof of this proposition is rather standard and will be given later in \cref{pf:prop:block unitary homotopy}. We thus further explain how to find the special projection $P$ and how to setup the homotopy \cref{eq:unitary first homotopy}.

The following lemma gives a concrete class of projections, diagonal in the position basis, that are ``sufficiently infinite'' in the sense of \cref{eq:infintie projections}, while still being ``small'' enough for us to construct a homotopy \cref{eq:unitary first homotopy}.

\begin{lem}\label{lem:class of infinite projections}
Let $S\subseteq\ZZ^2$ be such that for any $z\in\bS^1_\QQ$, there exists $x\in S$ such that $\arg(x) =\arg(z)$. Then $\Lambda_S$ (as in \cref{def:proj onto subsets of Z2}) satisfies \cref{eq:infintie projections}.
\end{lem}

We consider a sequence $\{x_k\}_{k\in\NN}$ of ``centers'' on $\ZZ^2$ that satisfies the premise of \cref{lem:class of infinite projections}. For a given $L$-local operator $U$, we would like the support of $U \delta_{x_k}$ to be around  the point $x_k$ itself.
By removing matrix elements of $U$ (with respect to the standard basis on $\ell^2(\ZZ^2)$) that are small in operator norm, we may deform $U\sim_h G$ via straight-line homotopy to an invertible operator $G$ where the action of $G$ on $\delta_{x_k}$ is around $x_k$.
We can then build a $L$-local operator $V$ that reverts the action of $G\delta_{x_k}$ back to $\delta_{x_k}$. In other words, by construction, the operator $VG$ would act as identity on the centers $\{x_k\}_{k\in\NN}$. 
Moreover, since the operator $V$ acts non-trivially only around the centers $\{x_k\}_{k\in\NN}$, it would follow that the path $V\sim_h \Id$ can be constructed straightforwardly.
This leads to
\eql{\label{eq:homotopy with V}
    U\sim_h G \sim_h VG \sim_h  P + P^\perp W P^\perp
}
which yields the homotopy \cref{eq:unitary first homotopy}. The last hotomopy in \cref{eq:homotopy with V} entails some minor manipulation to get to the final form.

Here we describe the centers more precisely.

\begin{lem}\label{lem:localized centers}
Let $A$ be $L$-local and $\ve>0$. Let $\{\theta_k\}_{k\in\NN}$ be a sequence of radians of $\bS^1_\QQ$. Then there exists a $L$-local operator $B$ such that $\|A-B\|\leq \ve$, and a sequence of centers $\{x_k\}_{k\in\NN}$ with $\arg(x_k)=\theta_k$ for all $k\in\NN$ such that if we define
\eql{\label{eq:interaction range of B}
    Y_k := \supp(B \delta_{x_k}) \cup \supp(\delta_{x_k}) \subset \ZZ^2
 }
to be the interaction range of $B$ on $\delta_{x_k}$, then $\{Y_k\}_{k\in\NN}$ satisfies the following.
\begin{enumerate}
\item $|Y_k|<\infty$ for all $k\in\NN$;
\item $\{Y_k\}_{k\in \NN}$ are pairwise disjoint.
\item For any two closed disjoint intervals $I,J$ of $\bS^1_\QQ$, we have
\eql{\label{eq:finite mixing between two}
    |\bigcup \Set{Y_k| k\in\NN,\ Y_k\cap C_I\neq \varnothing,\ Y_k\cap C_J\neq \varnothing}| < \infty \,.
}
\end{enumerate}
\end{lem}

As a corollary to \cref{lem:localized centers}, we can construct $V$ as promised in the homotopy \cref{eq:homotopy with V}.

\begin{cor}\label{cor:V operator that revert localized centers}
Let $B$ be a $L$-local operator, and $\{x_k\}_{k\in\NN}$ be centers for $B$ such that $Y_k$ as defined in \cref{eq:interaction range of B} satisfies the items 1, 2, and 3 in \cref{lem:localized centers}. 
Then there exists $V\in\calU(\calL_L)$ such that $V\sim_h\Id$ and
\eql{\label{eq:VB block form}
    VB = P + P VB P^\perp + P^\perp VB P^\perp \,.
}
\end{cor}



\begin{proof}[Proof of \cref{thm:R-local unitaries are path-connected}]
Let $U\in\calU(\calL_L)$ and $\ve$ be a small positive number. Let $\{\theta_k\}_{k\in\NN}$ be a sequence of radians in $\bS^1_\QQ$ that enumerate $\bS^1_\QQ$, i.e., for each $\theta\in\bS^1_\QQ$, there exists exactly one $k\in\NN$ such that $\theta_k=\theta$. 
Using \cref{lem:localized centers}, there exists a $L$-local operator $G$ such that $\|U-G\|\leq \ve$, and there exists localized centers $\{x_k\}_{k\in\NN}$ such that the subsets $Y_k$ as defined in \cref{eq:interaction range of B} has those properties specified in the lemma. In particular, we can choose $\ve$ is small enough so that $G\in\calG(\calL_L)$ is invertible.

Now we use \cref{cor:V operator that revert localized centers} to construct $V\in\calU(\calL_L)$ such that $V\sim_h \Id$ and $VB$ satisfies \cref{eq:VB block form}. 
Now
\eq{
    VG =P + PVGP^\perp + P^\perp VG P^\perp = (P + P^\perp VG P^\perp)(\Id + PVGP^\perp)\,.
}
Since $VG\in\calL_L$ and $[P,L]=0$, it is clear that $P + P^\perp VG P^\perp$ and $\Id + PVGP^\perp$ are both $L$-local. Observe that $(\Id + t PVGP^\perp)$ is always in $\calG(\calL_L)$ for $t\in[0,1]$, with inverse $\Id - t PVGP^\perp$. Therefore, we have $VG\sim_h P + P^\perp VG P^\perp$ in $\calG(\calL_L)$. Let $W$ be the polar part of $P + P^\perp VG P^\perp$. It follows from \cite[Proposition 2.1.8]{rordam2000introduction} that we have $P + P^\perp VG P^\perp\sim_h W$ in $\calG(\calL_L)$. 
Now $W$ is of the form $W=P + P^\perp W P$ and hence, by \cref{prop:block unitary homotopy}, we have $W\sim_h \Id$. Therefore, we have $VG\sim_h \Id$ in $\calG(\calL_L)$. Combined with $V\sim_h \Id$, we have $G\sim_h\Id$ in $\calG(\calL_L)$, and hence $U\sim_h\Id$ in $\calG(\calL_L)$. Using \cite[Proposition 2.1.8]{rordam2000introduction} again, we obtain the homotopy $U\sim_h \Id$ in $\calU(\calL_L)$ and concludes the proof.
\end{proof}

\section{\texorpdfstring{The proof of \cref{thm:classification of R-non-trivial orthogonal projections}: classification of orthogonal projections}{}}\label{sec:orthogonal projections classification proof}

We now turn to \cref{thm:classification of R-non-trivial orthogonal projections}. We denote $\calP^\NT_L$ the space of $L$-non-trivial $L$-local projections analogous to \cref{defn:non-triviality}. The index of $P\in\calP(\calN_L)$ is $\calN_L(P):=\findex(\PP L)$ (now $L$ replaces $R$ in \cref{eq:index of R-local projections}). To prove the theorem, we need to show injectivity and surjectivity of the map \cref{eq:projection index}. The surjecitivity statement 
\begin{lem}\label{lem:projection pi0 surjective}
The map \cref{eq:projection index} is surjective.
\end{lem} \noindent whose proof will be presented later. It relies on the Weyl-von Neuman-Berg theorem, to switch from the $L$ representation to an $R^k$ representation, where now $R$ is the concrete bilateral right shift on $\ell^2(\ZZ)$ and $k\in\ZZ$.

We are thus left with injectivity, which is tantamount to the statement that if $P,Q$ are two $L$-non-trivial projections which have the same index \cref{eq:projection index}, then there is a continuous path within the space of $L$-non-trivial projections which connects them. The hypothesis $\findex(\PP L) = \findex(\QQ L)$ already implies, via the calculation of $K_0$ performed below \cref{prop:K0 group of L local algebra}, that $[P]_0=[Q]_0$ as classes in $K_0$. In fact more is true thanks to $L$-non-triviality, as the following lemma shows.
\begin{lem}\label{lem:unitary equivalence of L-non-trivial projections}
If $P,Q\in\calP^\NT_L$ are $L$-non-trivial and have the same $\calN_L(P)=\calN_L(Q)$, then $P\sim_u Q$ in $\calP(\calL_L)$.
\end{lem}
\begin{proof}
Using \cref{prop:K0 group of L local algebra}, if $\calN_L(P)=\calN_L(Q)$, then $P,Q$ belong to the same $K_0$ class. It follows from \cite[Proposition 3.1.7]{rordam2000introduction} that $P\sim_s Q$. More is true. Since $P,Q$ are $L$-non-trivial, it follows from \cite[Proposition 5.1.4]{higson2000analytic} that we have $P\sim Q$. On the other hand, we have $\calN_L(P^\perp)=-\calN_L(P)=-\calN_L(Q)=\calN_L(Q^\perp)$. Using the same argument as before, we have $P^\perp \sim Q^\perp$. Using \cite[Proposition 2.2.2]{rordam2000introduction}, we conclude that $P\sim_u Q$. 
\end{proof}

Armed with this lemma and \cref{thm:R-local unitaries are path-connected}, we now are able to connect any two non-trivial projections with a path within $L$-local projections. However, this is not the end of the story, since the path must be $L$-non-trivial also. To that end, we use
\begin{lem}\label{lem:L-non-trivial are path components}
If $P\in\calP(\calL_L)$ and $P\sim_h Q$ in $\calP(\calL_L)$ for some $Q\in\calP^\NT_L$, then $P\in\calP^\NT_L$.
\end{lem}
Now we are finally able to finish the
\begin{proof}[Proof of \cref{thm:classification of R-non-trivial orthogonal projections}]
Suppose $P,Q\in\calP^\NT_L$ have the same index $\calN_L(P)=\calN_L(Q)$. Using \cref{lem:unitary equivalence of L-non-trivial projections}, we have $P\sim_u Q$ in $\calP(\calL_L)$, i.e., there exists $U\in\calU(\calL_L)$ such that $P=U^\ast QU$. From \cref{thm:R-local unitaries are path-connected}, we know that $U\sim_h \Id$ in $\calU(\calL_L)$, and this provides the homotopy $P\sim_h Q$ in $\calP(\calL_L)$. The rest of the claim is immediate from \cref{lem:projection pi0 surjective} (surjectivity) and \cref{lem:L-non-trivial are path components} (injectivity). Indeed, any point in the path bceomes immediately $L$-non-trivial via \cref{lem:L-non-trivial are path components} and so the whole path.
\end{proof}

\section{Details of the proofs}\label{sec:details of the proof}

\subsection{Local unitaries}

\begin{proof}[Proof of \cref{lem:class of infinite projections}]

Fix $n\in\NN$. Let $\{y_i\}_{i=1}^\infty$ be an enumeration of lattice points specified by $\Lambda_S\oplus \Id_n$, i.e., the set $S$ together with $n$-copies of $\ZZ^2$. Define
\eq{
	I_k=[\arg(y_k)-1/2^k,\arg(y_k)+1/2^k]\,.
} 
Iteratively, for each $k\in \NN$, pick $x_k\in S$ such that $x_k$ minimizes
\eq{
    \Set{ \norm{x} | x\in S\setminus \{x_1,\dots,x_{k-1}\},~ \arg(x)\in I_k}\,.
}
This is possible since $I_k$ intersects $\Set{\arg(x)|x\in S}$ for infinitely many $x$. Consider the mapping $S\cup (\ZZ^2)^n \ni y_k\mapsto x_k \in S$. The map is injective since in each step we exclude previously chosen points. The map is surjective since we choose $x\in S$ that minimizes ts length. Let $V$ maps $\delta_{y_k}$ to $\delta_{x_k}$. Then $V^\ast V=\Lambda_S\oplus\Id_n$ and $VV^\ast = \Lambda_S$.
We now show that $V$ is local. Write 
\eq{
	V=\begin{bmatrix}V_0 & V_1 & \dots & V_n\end{bmatrix}\in M_{1,n+1}(\calB(\ell^2(\ZZ^2)))\,.
}
Let $l\in\{0,1,\dots,n\}$ and consider $V_l$. Let $I$ and $J$ be two disjoint closed intervals of $\bS^1_\QQ$. Let $\{y_{n_k}\}_{k=1}^\infty$ be the subsequence of $\{y_i\}_{i=1}^\infty$ that enumerate points in $\chi_I$ in the $l$-th stack (for $l=0$, the $l$-stack is $S$, and for $l\geq 1$, it is $\ZZ^2$). We argue that
\eq{
	|\Set{z_{n_k}\in J | k\in\NN}| < \infty\,.
}
Indeed, points between $\chi_I$ and $\chi_J$ must have angle larger than some fixed positive constant. However, the $z_k$ points are chosen such that $\arg(z_k)\in I_k$. In particular, we have $z_{n_k}\in I_{n_k}$ and hence $z_{n_k}$ must be arbitrarily close to points in $I$ in angles. Therefore, the operator $\chi_J V_l\chi_I$ is finite-rank.
    
\end{proof}

\begin{proof}[Proof of \cref{lem:localized centers}]
Let $A$ be $L$-local and $\ve>0$. Let $\{\theta_k\}_{k\in\NN}$ be a sequence of radians in $\bS^1_\QQ$. Let $\ve_k = \ve/2^{2k-1}$ for $k\in\NN$. Let $\{J_k\}_{k\in\NN}$ be a sequence of closed intervals of $\bS^1_\QQ$ such that the endpoints of $J_k$ enumerate all $\bS^1_\QQ$. For each $k\in \NN$, by \cref{cor:contained in cone}, there exists disjoint subsets $E^b_k$ and $E^g_k$ that partition $C_{J_k^c}\subset \ZZ^2$ such that $\|\Lambda_{E^b_k} A \Lambda_{J_k}\|\leq \ve_{2k-1}$, and $|E^g_k\cap C_I|<\infty$ for any closed interval $I$ of $\bS^1_\QQ$ that is disjoint from $J_k$. Define
\eql{\label{eq:cone elements}
    P_{2k-1} =\Lambda_{E_k^b},\ Q_{2k-1} = \Lambda_{E_{J_k}}\,.
}
Using \cref{cor:contained in annulus}, there exists a sequence $\{x_k\}_{k\in\NN}$ of points on $\ZZ^2$ such that $\arg(x_k)=\theta_k$ for each $k\in\NN$, and a sequence $\{r_k\}_{k\in\NN}$ of radii $0=:r_0<r_1<r_2<\dots$, such that $x_k\in B_{r_k}\setminus B_{r_{k-1}}$ and $\|\Lambda_{B_{r_k}^c\cup B_{r_{k-1}}}U (\delta_{x_k}\otimes \delta_{x_k}^\ast)\|\leq \ve_{2k}$ for all $k\in\NN$. Define
\eql{\label{eq:annulus elements}
    P_{2k}=\Lambda_{B_{r_k}^c\cup B_{r_{k-1}}},\ Q_{2k}=\delta_{x_k}\otimes \delta_{x_k}^\ast \,.
}
Then by \cref{lem:turn off elements}, there exists $B\in \calL_L$ such that 
\eql{\label{eq:turn off all elements}
    P_kB Q_k=0
}
for all $k\in\NN$ and $\|A-B\|\leq\ve$. 

Consider $\{Y_k\}_{k\in\NN}$ defined in \cref{eq:interaction range of B}. We argue that $Y_k\cap Y_l=\varnothing$ for all $k,l\in\NN$ such that $k\neq l$.
Indeed, if $y\in Y_k$, then $y=x_k$ or $y\in \supp(B\delta_{x_k})$. In the latter case, we have
\eq{
    \ip{\delta_y}{B\delta_{x_k}} = \ip{\delta_y}{(P_{2k}+P_{2k}^\perp) B (Q_{2k}+Q_{2k}^\perp)\delta_{x_k}} 
    =\ip{\delta_y}{P_{2k}^\perp B Q_{2k} \delta_{x_k}} 
    =\ip{P_{2k}^\perp \delta_y}{B\delta_{x_k}}
}
where we used \cref{eq:turn off all elements} and \cref{eq:annulus elements} in the second equality. Since $P_{2k}^\perp= \Lambda_{B_{r_k}\setminus B_{r_{k-1}}}$, if follows that $y$ must be in $B_{r_k}\setminus B_{r_{k-1}}$ for $\ip{\delta_y}{B\delta_{x_k}}$ to be nonzero, and hence $\supp(B\delta_{x_k})\subset B_{r_k}\setminus B_{r_{k-1}}$. Thus $Y_k\subset B_{r_k}\setminus B_{r_{k-1}}$ and they are pairwise disjoint since $B_{r_k}\setminus B_{r_{k-1}}$ are. Furthermore, since $|B_{r_k}\setminus B_{r_{k-1}}|<\infty$, it follows that $|Y_k|<\infty$.

We now show that for any two closed disjoint intervals of $I,J$ of $\bS^1_\QQ$, we have
\eql{\textstyle
	\left|\left(\bigcup_{\substack{k\in\NN \\ \arg(x_k)\in J}} Y_k \right) \cap C_I \right| <\infty \,.
    \label{eq:finite one directional interaction}
}
Indeed, the set $\bigcup_{\substack{k\in\NN \\ \arg(x_k)\in J}} Y_k$ is the union of $\bigcup_{\substack{k\in\NN \\ \arg(x_k)\in J}} \supp(B\delta_{x_k})$ and $\bigcup_{\substack{k\in\NN \\ \arg(x_k)\in J}} \{x_k\}$, and since the latter one is clearly disjoint from $C_I$, it suffices to show that the former set intersects $C_I$ for finitely many points. To that end, we have $J_l=J$ for some $l\in\NN$. We argue that $\bigcup_{\substack{k\in\NN \\ \arg(x_k)\in J}} \supp(B\delta_{x_k})\subset E^g_l\cup C_J$. 
Indeed, if $y\in \bigcup_{\substack{k\in\NN \\ \arg(x_k)\in J}} \supp(B\delta_{x_k})$, then $\ip{\delta_y}{B\delta_x}\neq 0$ for some $x\in\ZZ^2$ with $\arg(x)\in J$. Now
\eq{
    \ip{\delta_y}{B\delta_x} = \ip{\delta_y}{(\Lambda_{E^g_l}+\Lambda_{E^b_l}+\Lambda_J)B(\Lambda_{J^c}+\Lambda_J)\delta_x} 
    = \ip{\delta_y}{(\Lambda_{E^g_l}+\Lambda_J)B \Lambda_J\delta_x} 
    = \ip{(\Lambda_{E^g_l}+\Lambda_J)\delta_y}{B\delta_x}
}
where we used \cref{eq:turn off all elements} and \cref{eq:cone elements} in the second equality. Thus, we must have $y\in E_l^g\cup C_J$. Now, we have $|(E_l^g\cup C_J)\cap C_I| = |E^l_g\cap C_I|<\infty$, and conclude \cref{eq:finite one directional interaction}.

Finally, we show that for any two closed disjoint intervals $I,J$ of $\bS^1_\QQ$, we have \cref{eq:finite mixing between two}. Since each $|Y_k|$ is finite, it suffices to show that the set of indices $k\in\NN$ such that $Y_k\cap C_I\neq \varnothing$ and $Y_k\cap C_J\neq \varepsilon$ is finite. Suppose, by contradiction, there are infinitely many such indices, denoted $\{k_j\}_{j\in\NN}$. Note that the disjoint closed intervals $I$ and $J$ partition $\bS^1$ into four region (intersecting at endpoints), denoted as $I,J,K,\ti{K}$ where $K,\ti K$ are closed intervals of $\bS^1_\QQ$. Since there are infinitely many points $\{x_{k_j}\}_{j\in\NN}$ and only four regions $C_I,C_J,C_K,C_{\ti K}$ partitioning the $\ZZ^2$ plane, one of the region must contain infinitely many $x_{k_j}$ points. Suppose $x_{k_j}\in C_I$ for all $j\in\NN$ (after relabeling), then $|\bigcup_j Y_{k_j} \cap C_J|=\infty$. However, this contradicts \cref{eq:finite one directional interaction}. Indeed, we have
\eq{\textstyle
    |\bigcup_j Y_{k_j} \cap C_J|\leq \left|\left(\bigcup_{\substack{k\in\NN \\ \arg(x_k)\in I}} Y_k \right) \cap C_J \right| <\infty\,.
}
Similarly if $x_{k_j}\in C_J$ for all $j\in\NN$, we will reach a contradiction. Suppose $x_{k_j}\in C_K$ for all $j\in\NN$. We split the closed interval $K$ into two closed intervals $K_1,K_2$ of $\bS^1_\QQ$. Suppose, without loss of generality, that $x_{k_j}\in C_{K_1}$ for all $j\in\NN$ and $K_1$ is the part that is closer to $I$. Then we have $|\bigcup_j Y_{k_j} \cap C_J|=\infty$. However, this contradicts \cref{eq:finite one directional interaction} using the disjoint intervals $K_1$ and $J$. This concludes the proof of \cref{eq:finite mixing between two}.
\end{proof}

\begin{proof}[Proof of \cref{cor:V operator that revert localized centers}]

Define the unitary operator $V$ as follows: $V$ maps $B \delta_{x_k}$ to $\delta_{x_k}$; extend $B \delta_{x_k}$ to an orthonormal basis in $\im \Lambda_{Y_k}$; map those basis elements (except $B \delta_{x_k}$) to the position basis elements in $\im \Lambda_{Y_k}$ (except $\delta_{x_k}$); define $V$ to be identity on $ (\bigoplus_{k\in\NN} \im \Lambda_{Y_k})^\perp $. 
The operator $V$ is well-defined since each $|Y_k|$ is finite and $\{Y_k\}_{k\in\NN}$ consists of mutually disjoint subsets as shown in \cref{lem:localized centers}.

We argue that $V$ is $L$-local. To that end, we show that for any $I,J$ disjoint closed intervals of $\bS^1_\QQ$, we have that $\Lambda_J V \Lambda_I$ is finite-rank. We can decompose $\ZZ^2$ into four disjoint subsets
\eq{\textstyle
    C_{I^c},\ C_I\cap (\ZZ^2 \setminus \bigcup_{k\in\NN} Y_k),\ C_I\cap Z^g,\ C_I\cap Z^b
}
where
\eq{\textstyle
Z^g &= \bigcup\Set{Y_k|k\in\NN, Y_k\cap C_I\neq \varnothing,\ Y_k\cap C_J\neq \varnothing} \\
Z^b &= \bigcup \Set{Y_k|k\in\NN, (Y_k\cap C_I= \varnothing ~\mathrm{or}~ Y_k\cap C_J= \varnothing)}\,.
}
Clearly $\Lambda_JV\Lambda_I\Lambda_{I^c}=0$. Since $V$ is identity on $(\bigoplus_{k\in\NN} \im \Lambda_{Y_k})^\perp $, it follows that \eq{
\Lambda_JV\Lambda_I\Lambda_{C_I\cap (\ZZ^2 \setminus \bigcup_{k\in\NN} Y_k)}=\Lambda_J\Lambda_{C_I\cap (\ZZ^2 \setminus \bigcup_{k\in\NN} Y_k)}=0\,.
}
Suppose $y\in C_I\cap Z^b$, then either $y\in C_I$ and $y\in Y_k$ for some $k\in\NN$ such that $Y_k\cap C_I=\varnothing$, or $y\in C_I$ and $y\in Y_k$ for some $k\in\NN$ such that $Y_k\cap C_J=\varnothing$. The former case is vacuous. In the latter case, we have $\Lambda_JV\Lambda_I\delta_y=\Lambda_JV\delta_y$; and since $V\delta_y\in \im \Lambda_{Y_k}$ and $Y_k\cap C_J=\varnothing$, it follows that $\Lambda_JV\delta_y=0$. Using \cref{eq:finite mixing between two}, the set $C_I\cap Z^g$ is finite and hence $\Lambda_JV\Lambda_I$ is finite-rank.

The argument in the previous paragraph works for all unitary operators that are reduced by projections $\Lambda_{Y_k}$ for all $k\in \NN$, and are identity on $(\bigoplus_{k\in\NN} \im \Lambda_{Y_k})^\perp $. Therefore, we can construct $V\sim_h \Id$ by deforming each unitary matrices restricted on each $\im \Lambda_{Y_k}$.

It is immediate by construction that $VB$ take the form \cref{eq:VB block form}.
\end{proof}

The following proof is inspired by \cite[Exercise 8.11]{rordam2000introduction}.
\begin{proof}[Proof of \cref{prop:block unitary homotopy}]\label{pf:prop:block unitary homotopy}

By \cref{cor:consequence of trivial K_1 group for L-local algebra}, there exists some $n\in\NN$ such that  $U\oplus \Id_n\sim_h \Id_{n+1}$. Let $W_t\in\calU_{n+1}(\calL_L)$ be such homotopy, with $W_0=\Id_{n+1}$ and $W_0=U\oplus \Id_n$. 
Now invoking \cref{eq:infintie projections} on our $P$ (by hypothesis) with that same $n$, we have $P\oplus\Id_n\sim_0 P$, i.e., there exists a partial isometry $T\in M_{1,n+1}(\calL_L)$ such that $|T|^2\equiv T^\ast T=P\oplus \Id_n$ and $|T^\ast|^2\equiv T T^\ast= P$. Using the identity $T=TT^\ast T$, we have
\eq{
    P^\perp T=P^\perp (TT^\ast) T = 0
}
and 
\eq{
T \begin{bmatrix}P^\perp \\ 0_{n,1}\end{bmatrix} = T (T^\ast T) \begin{bmatrix}P^\perp \\0_{n,1}\end{bmatrix} = 0
}
Now define
\eq{
	V := \begin{bmatrix}P^\perp & 0_{1,n}\end{bmatrix} + T \in M_{1,n+1}(A)\,.
}
We claim $V$ is an isometry. Indeed,
\eq{
	V^\ast V = (\begin{bmatrix}P^\perp \\ 0_{n,1}\end{bmatrix} + T^\ast)(\begin{bmatrix}P^\perp & 0_{1,n}\end{bmatrix} + T) = P^\perp\oplus 0_n + T^\ast T = \Id\,.
}
On the other hand
\eq{
	VV^\ast = (\begin{bmatrix}P^\perp & 0_{1,n}\end{bmatrix} + T)(\begin{bmatrix}P^\perp \\ 0_{n,1}\end{bmatrix} + T^\ast) = P^\perp +TT^\ast=\Id\,.
}
Consider the continuous path
\eq{
	Z_t = VW_tV^\ast + (\Id-VV^\ast)\,.
}
Using $V^\ast (\Id-VV^\ast)=0$ and $V^\ast V=\Id$, it follows that $Z_t\in\calU(\calL_L)$. Indeed, we have
\eq{
    Z_t^\ast Z_t =  (VW_t^\ast V^\ast + (\Id-VV^\ast))  (VW_tV^\ast + (\Id-VV^\ast)) =VV^\ast + (\Id-VV^\ast)^2 = \Id
}
and similarly
\eq{
    Z_tZ_t^\ast = (VW_tV^\ast + (\Id-VV^\ast)) (VW_t^\ast V^\ast + (\Id-VV^\ast)) = VV^\ast + (\Id-VV^\ast)^2=\Id.
}
It is clear that $Z_0=\Id$. We now show that $Z_1=U$. Indeed, we need to compute $V (U \oplus \Id_n)V^\ast + (\Id-VV^\ast)$. Consider
\eq{
	V (U\oplus \Id_n) &= (\begin{bmatrix}P^\perp & 0_{1,n}\end{bmatrix} + TT^\ast T) (U\oplus \Id_n) \\
    &=\begin{bmatrix}P^\perp U & 0_{1,n}\end{bmatrix} + T(P\oplus \Id_n)(U\oplus \Id_n) \\
    &= \begin{bmatrix}P^\perp U P^\perp &0_{1,n}\end{bmatrix} + T 
}
where in the last equality, we used $PU=P$ via \cref{eq:U acts as identity within an infinite projection}. Proceeding, we have
\eq{
	(\begin{bmatrix}P^\perp U P^\perp &0_{1,n}\end{bmatrix} + T)  (\begin{bmatrix}P^\perp \\ 0_{n,1}\end{bmatrix} + T^\ast) = P^\perp U P^\perp + TT^\ast
}
and finally, we obtain
\eq{
V (U \oplus \Id_n)V^\ast + (\Id-VV^\ast) = P^\perp U P^\perp + TT^\ast + (\Id-VV^\ast) = P^\perp U P^\perp + TT^\ast +P -TT^\ast \stackrel{\cref{eq:U acts as identity within an infinite projection}}{=} U.
}
\end{proof}

\subsection{Local projections}

\begin{proof}[Proof of \cref{lem:projection pi0 surjective}]
Consider the Hilbert space $\ell^2(\ZZ)$ and let $R$ be the bilateral right shift, i.e., $R\delta_x = \delta_{x+1}$ for all $x\in\ZZ$. It is clear that $\Lambda R\Lambda + \Lambda^\perp$ is Fredholm of index $-1$, where $\Lambda$ here is $\Lambda:=\sum_{x\geq 1}\delta_x\otimes \delta_x^\ast$.

Let $f\in C(\bS^1)$. We argue that if $\Lambda f(R)\Lambda\in\calK$, then $f=0$. To that end, we first observe that 
\eq{
\|f(R)\delta_x\|=\|f(R)\delta_y\|
}
for all $x,y\in\ZZ$. Indeed, let $p_n$ be a sequence of polynomials in two variables that converge uniformly to $f$. Then 
\eq{
\|f(R)\delta_x\|=\lim_n\|p_n(R,R^\ast)\delta_x\|=\lim_n\|p_n(R,R^\ast)\delta_y\|=\|f(R)\delta_y\|\,.
}
Therefore, there are two cases: either $\|f(R)\delta_x\|=0$ for all $x\in\ZZ$ or $\|f(R)\delta_x\|=C>0$ for all $x\in\ZZ$.
In the former case, we have $f(R)=0$ and hence $f=0$ and we are done. Suppose it is the latter case. For $x$ large enough, we have
\eq{
    \|\Lambda f(R)\Lambda \delta_x\| = \|\Lambda f(R)\delta_x\| \geq \|f(R)\delta_x\| - \|[\Lambda,f(R)]\delta_x\| \,.
}
By assumption that $\Lambda f(R)\Lambda$ is compact, it follows that $\|\Lambda f(R)\Lambda \delta_x\|\to 0$ as $x\to\infty$. However, the commutator $[\Lambda,f(R)]$ is also compact and we have $\|[\Lambda,f(R)]\delta_x\|\to 0$ as $x\to\infty$. This leads to a contractdiction since $\|f(R)\delta_x\|=C>0$ for all $x\in\ZZ$.

Using the Weyl-von Neumann-Berg theorem \cite[Theorem 39.8]{conway2000course}, there exists a unitary operator $U:\ell^2(\ZZ^2)\to\ell^2(\ZZ)$ such that
\eq{
    L=U^\ast RU + K
}
for some $K\in\calK(\ell^2(\ZZ^2))$. Since $\Lambda$ is $R$-local, it follows that $P:=U^\ast \Lambda U$ is $L$-local, and that $\findex(PLP+P^\perp)=-1$, see also \cite[Proposition 2.1]{chung2024topological}. 
We show that $P$ is $L$-non-trivial. Indeed, let $f\in C(\bS^1)$ and suppose $Pf(L)P\in\calK$. Then
\eq{
Pf(L)P=U^\ast \Lambda U f(L) U^\ast \Lambda U = U^\ast \Lambda  f(ULU^\ast) U^\ast \Lambda U\,.
}
Now $f(ULU^\ast)-f(R)$ is compact since $ULU^\ast-R=\ti K$ is. Indeed, with $\pi:\calB\to\calQ$ being the quotient map to the Calkin algebra, we have 
\eq{
    \pi(f(ULU^\ast))=\pi(f(R+\ti K)) = f(\pi(R+\ti K)) = f(\pi(R)) = \pi (f(R))
}
where we used \cite[Proposition 4.4.7]{kadison_fundamentals_1997} to swap $\pi$ and $f$. It follows that $\Lambda f(R)\Lambda$ is compact and hence $f=0$. In an analogous way, if $f\in C(\bS^1)$ and $P^\perp f(L) P^\perp\in\calK$, then $\Lambda^\perp f(R)\Lambda^\perp\in\calK$ and hence $f=0$. This concludes the argument that $P$ is $L$-non-trivial.

In summary, we have obtained a $L$-non-trivial projection $P$ that has $\calN_L(P)=-1$. To obtain $P\in\calP^\NT_L$ that has $\calN_L(P)=k$ for some other $k\in\ZZ$, we can repeat the argument by considering the shift $R^{-k}$ operator on $\ell^2(\ZZ)$.
\end{proof}

\begin{proof}[Proof of \cref{lem:L-non-trivial are path components}]
Let $P\in\calP(\calL_L)$ and suppose we have $P\sim_h Q$ for some $Q\in\calP^\NT_L$. By \cite[Proposition 2.2.7]{rordam2000introduction}, we know that $P\sim_h Q$ implies $P\sim_u Q$, and hence there is $U\in\calU(\calL_L)$ such that $P=U^\ast QU$. Let $f\in C(\bS^1)$. Suppose $Pf(L)P\in\calK$. We have
\eq{
    Qf(L)Q = UPU^\ast f(L) UPU^\ast =  UPU^\ast [f(L) ,U]PU^\ast + UPf(L)PU^\ast \in\calK \,.
}
By $L$-non-triviality of $Q$, we have that $f=0$. Using $P^\perp\sim_uQ^\perp$, we can show similarly that if $P^\perp f(L)P^\perp\in\calK$, then $f=0$.
\end{proof}

	\bigskip
	\bigskip
	\noindent\textbf{Acknowledgments.} 
	We are  indebted to Christopher Broune, Gian Michele Graf, Shanshan Hua and Jonathan Rosenberg for stimulating discussions.
	\bigskip

\appendix
\section{Locality and boundedness}\label{sec:locality and boundedness}
We recount an equivalent way to formualte $L$-locality, which was discussed in \cite{chung2024topological}.
\begin{thm}[Theorem 2.5 of \cite{chung2024topological}]\label{eq:cone characterization of L locality}
An operator $A$ is $L$-local iff for any two disjoint closed interval $I,J\subset \bS^1_\QQ$, we have
\eql{
    \Lambda_I A \Lambda_J \in\calK \,.
}
\end{thm}

We note the technique in \cite[Theorem 2.1]{gilfeather1983commutants} provides another proof of \cref{eq:cone characterization of L locality}, which is different from the proof given in \cite[Theorem 2.5]{chung2024topological}.

\begin{lem}\label{lem:approx compact by finite projection}
Let $K\in \calK(\calH)$ be compact and $E\subset \ZZ^2$. For any $\ve>0$, there exists a finite subset $F_\ve\subset E$ such that
\eq{
    \|\Lambda_{F_\ve} K - \Lambda_E K\|\leq \ve.
}
\end{lem}
\begin{proof}
If $E$ is finite, then we are can choose $F$ to be $E$. 
Otherwise, consider an increasing sequence of finite subsets $F_1\subset F_2\cdots$ such that $\bigcup_{k\in\NN}F_k=E$. Then $\Lambda_{F_k}$ converges to $\Lambda_E$ in the strong operator topology. Since $\Lambda_E K$ is compact, it follows that $\Lambda_{F_k}\Lambda_EK=\Lambda_{F_k}K$ converges in norm to $\Lambda{E}\Lambda_EK=\Lambda_EK$. Therefore, for $k$ large enough, we have $\|\Lambda_{F_k}K-\Lambda_EK\|\leq \ve$.
\end{proof}

\begin{cor}\label{cor:contained in cone}
Let $A$ be $L$-local and $\ve>0$ be arbitrary. Let $J\subset\bS^1_\QQ$ be a closed interval. Then there exist two disjoint subsets $E^g$ and $E^b$ partitioning $C_{J^c}$ such that 
\eq{
	\|\Lambda_{E^b} A \Lambda_J\|\leq \ve
}
and for any closed interval $I\subset \bS^1_\QQ$ disjoint from $J$, we have
\eq{
	|E^g\cap C_I|<\infty \,.
}
\end{cor}
\begin{proof}
Let $J=[\theta_1,\theta_2]$ be the closed interval subset of $\bS^1_\QQ$ traced from $\theta_1$ to $\theta_2$ counterclockwise. Let 
\eq{
	N_k=[\theta_1-1/2^k,\theta_2+1/2^k]
}
be the intervals slightly larger than $J$ and shirnks toward $J$. By \cref{eq:cone characterization of L locality}, we have $\Lambda_{N^c_k}A\Lambda_J \in \calK$. Using \cref{lem:approx compact by finite projection}, we can decompose $C_{N^c_k}$ into two disjoint subsets
\eq{
	C_{N^c_k} = E^g_k \cup E^b_k
}
such that $|E^g_k|<\infty$ and $\| \Lambda_{E_k^b} A \Lambda_J\|\leq \ve/2^k$. 
Let
\eq{
	E^b = \cup_{k\in\NN} E_k^b \,.
}
We can rewrite $E^b$ as $\cup_{k\in\NN} F_k$ where $F_k = E_k^b\setminus (\cup_{i=1}^{k-1}E_i^b)$ are disjoint. Thus
\eq{
	\|\Lambda_{E^b} A \Lambda_J\| &= \|\sum_{k=1}^\infty \Lambda_{F_k} A \Lambda_J\| \leq \sum_{k=1}^\infty \| \Lambda_{F_k} A \Lambda_J\| \leq \sum_{k=1}^\infty \| \Lambda_{E_k^b} A \Lambda_J\| \leq \ve\,.
}
Define the good set as
\eq{
	E^g = C_{J^c}\setminus E^b\,.
}
We argue that $|E^g\cap C_{N_k^c}|<\infty$ for all $k\in\NN$. Indeed, we have
\eq{
    E^g\cap C_{N_k^c} = C_{N_k^c}\cap (E^b)^c = (E_k^g\cup C_k^b) \cap (E^b)^c = E_k^g\cap (E^b)^c \subset E_k^g\,.
}
If $I$ is a closed interval of $\bS^1_\QQ$, then there exists $k\in\NN$ such that $I\subset C_{N_k^c}$. It follows that $|E^g\cap I|<\infty$.
\end{proof}



The next result is a consequence of boundedness and does not require the locality of the operator.
\begin{lem}\label{cor:contained in annulus}
Let $A$ be a bounded operator on $\calH$ and $\{\ve_i\}_{i=1}^\infty$ be sequence of positive numbers. Let $\{\theta_i\}_{i=1}^\infty$ be a sequence of radians on $\bS^1_\QQ$. Then there exists a sequence $\{x_i\}_{i=1}^\infty$ of points on $\ZZ^2$ such that $\arg(x_i)=\theta_i$ for $i\in\NN$, and a sequence $\{r_i\}_{i=1}^\infty$ of radii $0=:r_0<r_1<r_2<\dots$, such that $x_i\in B_{r_i}\setminus B_{r_{i-1}}$, and 
\eq{
	\|\Lambda_{B_{r_i}^c\cup B_{r_{i-1}}}A (\delta_{x_i}\otimes \delta_{x_i}^\ast)\| \leq \ve_i\,.
}
In other words, the effect of $A$ on $\delta_{x_i}$ is confined in the region $B_{r_i}\setminus B_{r_{i-1}} = \ZZ^2\setminus (B_{r_i}^c\cup B_{r_{i-1}})$.
\end{lem}
\begin{proof}
Let us pick an element $x_1\in\ZZ^2$ such that $\arg(x_1)=\theta_1$. There exists $r_1>\|x_1\|$ such that 
\eq{
\|\Lambda_{B_{r_1}^c} A \delta_{x_1}\otimes (\delta_{x_1})^\ast\|\leq \ve_1 \,.
}
Indeed, this follows from \cref{lem:approx compact by finite projection} with the fact that $A \delta_{x_1}\otimes (\delta_{x_1})^{\ast}$ is compact and $\Lambda_{B_{r}}$ converges to $\Id$ strongly as $r\to\infty$. 
Since $|B_{r_1}|$ is finite, which implies $A^\ast \Lambda_{B_{r_1}}$ is compact, it follows from \cref{lem:approx compact by finite projection} again that there exists $t_1>r_1$ such that $\|\Lambda_{B_{t_1}^c}A^\ast\Lambda_{B_{r_1}}\|\leq \ve_2/2$, or
\eq{
    \| \Lambda_{B_{r_1}} A \Lambda_{B_{t_1}^c}\|\leq \ve_2/2 \,.
}
Pick $x_2\in\ZZ^2$ with $\|x_2\|> t_1$. There exists $r_2>\|x_2\|$ such that
\eq{
    \|\Lambda_{B_{r_2}^c} A (\delta_{x_2}\otimes \delta_{x_2}^\ast)\| \leq \ve_2/2
}
which follows again from \cref{lem:approx compact by finite projection}. Therefore, we have
\eq{
    \|\Lambda_{B_{r_2}^c\cup B_{r_1}} A (\delta_{x_2}\otimes \delta_{x_2}^\ast)\| \leq \|\Lambda_{B_{r_2}^c} A (\delta_{x_2}\otimes \delta_{x_2}^\ast)\| + \|\Lambda_{B_{r_1}} A \Lambda_{B_{t_1}^c}\| \leq \ve_2
}
where we used $\delta_{x_2}\otimes \delta_{x_2}^\ast\leq \Lambda_{B_{t_1}^c}$ in the first inequality. 

We iterate the procedure: we pick $t_2>r_2$ such that $\|\Lambda_{B_{t_2}^c} A^\ast \Lambda_{B_{r_2}}\|\leq \ve_3/2$; pick $x_3\in\ZZ^2$ with $\|x_3\|>t_2$; and pick $r_3>\|x_3\|$ such that $\|\Lambda_{B_{r_3}^c} A (\delta_{x_3}\otimes \delta_{x_3}^\ast)\| \leq \ve_3/2$; and we get $\|\Lambda_{B_{r_3}^c\cup B_{r_2}} A (\delta_{x_3}\otimes \delta_{x_3}^\ast)\|\leq \ve_3$; and so on.
\end{proof}

\cref{cor:contained in annulus} tries to control the effect of a bounded operator on a lattice point to an annulus, using only the boundedness of the operator, i.e., without using $L$-locality.

We would like the remove all those hopping $PAQ$ from $\im Q$ to $\im P$ when $\|PAQ\|$ is small. When $P,Q\subset\ZZ^2$, we are turning off matrix elements in the infinite matrix $A_{xy}:=\ip{\delta_x}{A\delta_y}$. To achieve this, we need the following lemma. 

\begin{lem}\label{lem:turn off elements}
Let $A$ be $L$-local and $\ve>0$ be arbitrary. Let $\{P_k\}_{k=1}^\infty$ and $\{Q_k\}_{k=1}^\infty$ be two sequences of projections such that $[P_k,L]=[Q_k,L]=0$ for all $k\in\NN$ and
\eql{\label{eq:norm bound of Pk A Qk}
    \norm{P_k A Q_k}\leq \frac{\ve}{2^{2k-1}},~ \forall k\in\NN \,.
}
Then there exists a $L$-local operator $B$ such that $P_kB Q_k=0$ for all $k\in\NN$, and $\| A-B\|\leq \ve$.
\end{lem}
\begin{proof}
We would have liked to define $B$ as \eq{
    A - \sum_{i} P_i A Q_i\,.
} However, this formula may fail to represent the operator we want since the range of the projections $P_i,P_j$ or $Q_i,Q_j$ may overlap, which would mean we over-delete elements. To remedy this problem, inspired by the inclusion–exclusion formula, we define
\eql{
	S_n&:=\sum_{i=1}^n P_i AQ_i - \sum_{1\leq i<j\leq n}P_iP_j A Q_iQ_j \notag \\
	&\quad + \sum_{1\leq i<j<k\leq n}P_iP_jP_k A Q_iQ_jQ_k - \dots + (-1)^{n-1}P_1\dots P_n A Q_1\dots Q_n \notag \\
	&= \sum_{i=1} (-1)^{i+1}\left(\sum_{1\leq k_1<\dots<k_i\leq n} P_{k_1}\dots P_{k_i} A Q_{k_1}\dots Q_{k_i}\right) \label{eq:inclusion exclusion}
} which corrects all the over-counting, so that at least heuristically $B := A - S_\infty$. More precisely, let $\ve_k:=\|P_k AQ_k\|$. Then
\eql{\label{eq:norm bound on S_n}
	\|S_n\|\leq \sum_{k=1}^n 2^{k-1}\ve_k\,.
}
For example, we have $\|S_1\|=\|P_1 AQ_1\|=\ve_1$, and 
\eq{
\|S_2\|=\|P_1AQ_1+P_2 AQ_2-P_1P_2AQ_1Q_2\| \leq \ve_1+2\ve_2
}
where we used $\|P_1P_2AQ_1Q_2\|\leq \|P_2AQ_2\|=\ve_2$. Fix $l$, we count the number of terms $P_{k_1}\dots P_{k_l}AQ_{k_1}\dots Q_{k_l}$ in $S_n$ with $k_1<\dots <k_l$ having $k_l=m$. Then use the fact that
\eq{
	\|P_{k_1}\dots P_{k_l}AQ_{k_1}\dots Q_{k_l}\|\leq \|P_{k_l}AQ_{k_l}\|=\ve_m \,.
}
There are $2^{m-1}$ number of terms of that form. Let $S=\lim_{n\to\infty} S_n$. We need to show that the limit exists. To that end, for $m>n$, consider $S_m-S_n$. Using the formula \cref{eq:inclusion exclusion} and idea leading to upper bound \cref{eq:norm bound on S_n}, all the terms $P_{k_1}\dots P_{k_l}AQ_{k_1}\dots Q_{k_l}$ in $S_m-S_n$ will have $k_l\geq n+1$. Using \cref{eq:norm bound of Pk A Qk}, we have
\eq{
	\|S_m-S_n\| &\leq 2^n\ve_{n+1}+2^{n+1}\ve_{n+2}+\dots + 2^{m-1}\ve_m \\
	&\leq \frac{\ve}{2^{{n+1}}} +\dots + \frac{\ve}{2^m} \\
	&\leq \ve \sum_{k=n+1}^\infty \frac{1}{2^k}
}
which converges to $0$ as $n\to\infty$ (independent of $m$). 

We have $\|S\|\leq \ve$. Indeed, from \cref{eq:norm bound on S_n} and \cref{eq:norm bound of Pk A Qk}, we have
\eq{
	\|S\|\leq \sum_{k=1}^\infty 2^{k-1}\frac{1}{2^{k-1}}\frac{\ve}{2^k} \leq \ve \,.
}
Define
\eq{
    B = A-S \,.
}
Let us now show that
\eq{
	P_k B Q_k = 0,\ \forall k\in\NN
}
to prove that $B$ is indeed unaffected by the interactions originally present in $P_kAQ_k$. We show that $P_kS_nQ_k=P_kAQ_k$ for all $n\geq k$ by induction. From \cref{eq:inclusion exclusion}, we have the recursion relation
\eql{
	S_{n+1} = S_n + P_{n+1}AQ_{n+1} - P_{n+1}S_nQ_{n+1} \label{eq:recursion relation}
}
whre $P_{n+1}AQ_{n+1}$ is from the first sum in \cref{eq:inclusion exclusion}, and $P_{n+1}S_nQ_{n+1}$ is from the rest of sums. Let $k\geq 1$ be arbitrary. It holds that $P_1S_1Q_1=P_1AQ_1$. Take $n=k$ in \cref{eq:recursion relation} and consider
\eq{
	P_{k+1} S_{k+1} Q_{k+1} &= P_{k+1} S_kQ_{k+1} + P_{k+1}P_{k+1}AQ_{k+1}Q_{k+1} - P_{k+1}P_{k+1}S_kQ_{k+1}Q_{k+1} \\
	&= P_{k+1}AQ_{k+1} \,.
}
Thus $P_kS_kQ_k=P_kAQ_k$ holds for all $k\geq 1$. Suppose $P_kS_{n}Q_k=P_kAQ_k$ holds. Using \cref{eq:inclusion exclusion}, we have
\eq{
	P_kS_{n+1}Q_k &= P_kS_nQ_k + P_kP_{n+1}AQ_{n+1}Q_k - P_kP_{n+1}S_nQ_{n+1}Q_k \\
	&=P_kS_nQ_k + P_{n+1}P_kAQ_kQ_{n+1} - P_{n+1}P_kS_nQ_kQ_{n+1} \\
	&= P_kAQ_k
}
where in the last equality we used the induction assumption.

If $A$ is local, then $B=A-S$ is also local, since we are merely turning off some matrix elements (in the position basis). Indeed, by formula \cref{eq:inclusion exclusion}, we have $[S_n,L]\in\calK$ for all $n\geq 1$. Since $S_n\to S$ in norm, it follows that $S$ is also local.
\end{proof}

\section{\texorpdfstring{The $K$-theory of the local algebra}{The K-theory of local algebra}}
\label{sec:K-theory of the R-local algebra}

Let $\calC_L$ be the $C^\ast$-algebra generated by the Laughlin operator $L$ as given in \cref{subsubsection:concrete rep}. By the spectral theorem \cite[Theorem 4.30]{Douglas1998}, we have the isomorphism of $C^\ast$-algebras $\calC_L\cong C(\bS^1)$, where we used that $\sigma(L)=\bS^1$. There is a natural representation $\rho$ of $C(\bS^1)$ on $\calB(\ell^2(\ZZ^2))$ given by
\eq{
    C(\bS^1) &\cong \calC_L \hookrightarrow \calB(\ell^2(\ZZ^2))\\
    (z\mapsto z)&\mapsto L\,.
}
Furthermore, the representation $\rho$ is \emph{ample} in the sense that $\rho$ is unital and that $\rho(a)$ is compact only if $a=0$. Indeed, the only compact operator in $\calC_L$ is the zero operator. We can define the dual algebra of $C(\bS^1)$ with respect to $\rho$ as
\eql{\label{eq:dual of C(S)}
    \calD_\rho(C(\bS^1)):= \Set{A\in\calB(\calH)| [A,\rho(a)]\in\calK(\ell^2(\ZZ^2)),~\forall a\in C(\bS^1)}\,.
}
It is clear that, actually, we have
\eql{\label{eq:L local algeras is the dual algebra of C(S)}
    \calL_L = \calD_\rho(C(\bS^1)) \,.
}

To establish the next proposition, let us briefly define the $\Ext(X)$ commutative semigroup, where $X$ is a non-empty compact subset of $\CC$; see e.g. \cite[Definition 2.4.3]{higson2000analytic}. Consider the set of all essentially normal operators $T\in \calB(\calH)$ that have essential spectrum $X$. Here $\calH$ can be any separable Hilbert spaces. Two such operators $T_1\in \calB(\calH_1)$ and  $T_2\in \calB(\calH_2)$ are said to be essentially unitary equivalent iff there exists a unitary operator $U:\calH_1\to \calH_2$ such that $T_1-U^\ast T_2U\in\calK(\calH_1)$. We can quotient the set of all essentially normal operators having essential spectrum $X$ over the essentially unitary equivalence. For two class $[T_1]$ and $[T_2]$ where and $T_1,T_2$ are essentially normal operators in $\calB(\calH_1)$ and $\calB(\calH_2)$, respectively, we can consider $T_1\oplus T_2$ as operator in $\calB(\calH_1\oplus \calH_2)$. Then $T_1\oplus T_2$ is also an essentially normal operator with essential spectrum $X$. This addition operation makes the quotient set into a commutative semigroup denoted by $\Ext(X)$.

\begin{prop}\label{prop:K0 group of L local algebra}
We have an isomorphism $K_0(\calL_L)\cong \ZZ$ given by $[P]_0\mapsto \calN_L(P)\equiv \findex (PLP+P^\perp)$.
\end{prop}
\begin{proof}
Using \cite[Proposition 5.1.6]{higson2000analytic}, there is an isomorphism of abelian groups
\eql{\label{eq:K0 of C(S) and Ext group}
    K_0(\calD_\rho(C(\bS^1))) \cong \Ext(\bS^1)
}
where $\Ext(\bS^1)$ here is, in fact, an abelian group. The isomorphism \cref{eq:K0 of C(S) and Ext group} is given by
\eq{
    [P]_0 \mapsto [PLP:\im P\to \im P]\,.
}
On the other hand, the $\Ext(\bS^1)$ abelian group is isomorphic to $\ZZ$, where the isomohpism is given by $[A]\mapsto \findex(A)$; see \cite[Example 2.4.5 and Proposition 2.4.6]{higson2000analytic}. Thus, we have the isomorphism $K_0(\calL_L)\cong \ZZ$ given by $[P]_0\mapsto \findex(PLP+P^\perp)$.
\end{proof}

Let us now discuss the notion of a dual in somewhat more abstract terms. In general, for a separable, unital $C^\ast$-algebra $\calA$, we define its \emph{dual} algebra as
\eql{\label{eq:dual algebra def}
    \calD(\calA)=\Set{T\in\calB(\calH)|[T,\rho(a)]\in\calK(\calH),~\forall a\in\calA}
}
for any choice of an \emph{ample} representation $\rho:\calA\to \calH$ on some separable Hilbert space $\calH$, in the sense that $\rho$ is unital and $\rho(a)$ is compact only if $a=0$. See \cite[Definition 5.1.1,5.1.3]{higson2000analytic}. In fact, the dual algebra \cref{eq:dual algebra def} is independent of the choice of ample representation used to define it; see \cite[Section 5.2]{higson2000analytic}. Let $\calA$ be a separable $C^\ast$-algebra, possibly without unit, and let $\ti \calA$ be the $C^\ast$-algebra with a unit adjoined. \cite[Definition 5.2.7]{higson2000analytic} defines the \emph{$K$-homology} group of $\calA$ to be
\eql{\label{eq:def of K homology groups}
    K^0(\calA)=K_1(\calD(\ti \calA)),\quad K^1(\calA)=K_0(\calD(\ti\calA))\,.
}
In the case when $\calA$ is abelian, we have
\eql{\label{eq:Bott periodicity for K homology}
    K^0(\calS (\calA))\cong K^1(\calA),\quad K^1(\calA)=K^0(\calS(\calA))
}
where $\calS(\calA)$ is the suspension of $\calA$; see e.g. \cite[Eq. 7.2.7]{higson2000analytic}.

\begin{prop}\label{prop:K1 group of L local algebra}
The $K_1$ group of $\calL_L$ is trivial.
\end{prop}
\begin{proof}
We have
\eq{
    K_1(\calL_L)\cong K_1(\calD(C(\bS^1))) \equiv K^0(C_0(\RR)) \cong K^1(\CC)\equiv K_0(\calD(\ti \CC)) \cong 0
}
where in the first isomorphism, we used \cref{eq:L local algeras is the dual algebra of C(S)}; the second equivalence is by definition \cref{eq:def of K homology groups}, and the fact that $C(\bS^1)$ is the unitization of $C_0(\RR)$; the third isomorphism is by \cref{eq:Bott periodicity for K homology} and that $C_0(\RR)$ is the suspension of $\CC$; and the final isomorphism can be computed as in \cite[Example 5.2.9]{higson2000analytic}.
\end{proof}
\begin{cor}\label{cor:consequence of trivial K_1 group for L-local algebra}
    For any unitary $U\in\calU(\calL_L)$ there exists some $n\in\NN$ such that $U\oplus \Id_n \sim_h \Id_{n+1}$ in $\calU_{n+1}(\calL_L)$.
\end{cor}
\begin{proof}
This follows from the construction of the $K_1$-group of a unital $C^\ast$-algebra, see \cite[Definition 8.1.3]{rordam2000introduction}.
\end{proof}

\section{\texorpdfstring{$K$-theory and homotopy}{}}\label{sec:remark on K-theory and our main results}
In a vague sense, one could say, reading the above \cref{sec:K-theory of the R-local algebra}, that in this manuscript we show that $K_1(\calL_L)$ and  $\pi_0(\calU(\calL_L))$ agree for the $C^\ast$-algebra $\calL_L$ of operators which essentially commute with our fixed $L$, so that one might conjecture that studying $K$-theory should always be enough. In this appendix we discuss this idea: in general it is false that the $K_1(\calA)$ and $\pi_0(\calU(\calA))$ will always agree for a unital $C^\ast$-algebra $\calA$, but in fact, for certain algebras (which our $L$-local algebra is an example of) this result is indeed true and automatic. In this sense, this idea would lead to an alternative proof of \cref{thm:R-local unitaries are path-connected} presented above.

We thank Shanshan Hua \cite{private_shanshan} for useful discussions regarding what follows here.
\subsection{Counter-examples}\label{sec:counter-examples}

There are examples of $C^\ast$-algebras $\calA$ where $K_1(\calA)=\{0\}$ but $\pi_0(\calU(\calA))\neq \{0\}$, and also examples where $K_1(\calA)\neq \{0\}$ but $\pi_0(\calU(\calA))=\{0\}$.

For $n,m\geq 1$, we have
\eq{
K_1(M_n(C(\bS^m))) \cong K_1(C(\bS^m)) \cong \begin{cases} \ZZ & m~\mathrm{odd} \\ \{0\} & m~\mathrm{even}\end{cases}\,.
}
On the other hand
\eq{
\pi_0(\calU(M_n(C(\bS^m)))) \cong \pi_0(C(\bS^m,U(n)))\cong \pi_m(U(n))\,.
}
Take $m=3$ and $n=1$, then
\eq{
	K_1(C(\bS^3)) \cong \ZZ,\quad \pi_0(\calU(C(\bS^3)))\cong \pi_3(U(1))=\{0\}\,.
}
Take $m=4$ and $n=2$, then
\eq{
	K_1(M_2(C(\bS^4))) \cong \{0\},\quad \pi_0(\calU(M_2(C(\bS^4))))\cong \pi_4(U(2))=\ZZ_2\,.
}

\subsection{\texorpdfstring{$K_1$-injectivity}{}}

The $R$-local algebra $\calL_R$ is the same as the \emph{Paschke dual algebra} defined as
\eql{\label{eq:dual of C(S) in R rep}
\calL_R\equiv \calD_\rho(C(\bS^1)):=\Set{A\in\calB(\calH)| [A,\rho(a)]\in\calK(\calH),\, \forall a\in C(\bS^1)}
}
where $\rho: C(\bS^1)\to \calB(\calH)$ is the natural representation given by $C(\bS^1)\cong \calC_R\hookrightarrow \calB(\calH)$. See \cite[Chapter 5]{higson2000analytic} for more discussion about the dual algebra.
The $C^\ast$-algebra \cref{eq:dual of C(S) in R rep} is in fact $K_1$-injective in the sense that the group homomorphism
\eql{\label{eq:K1 injective group homomorphism}
    \calU(\calL_R)/\calU^0(\calL_R) \ni [U] \mapsto [U]_1 \in K_1(\calL_R)
}
is injective, see \cite[Sec 8.3]{rordam2000introduction} and \cite{hua2024k} for more detail about the group homomorphism \cref{eq:K1 injective group homomorphism}. The set $\calU^0(\calL_R)$ is the set of $U\in\calU(\calL_R)$ such that $U\sim_h\Id$ in $\calU(\calL_R)$.
\begin{proof}[Sketch of proof that \cref{eq:K1 injective group homomorphism} is injective due to Hua \cite{private_shanshan}] 
    First, one shows that the Paschke dual algebra \cref{eq:dual of C(S) in R rep} is properly infinite; see e.g. \cite[Lemma 2.2]{loreaux2020remarks}. Second, for unital properly infinite $C^\ast$-algebras, to show $K_1$-injectivity, it suffices to show that the natural map \eql{\label{eq:natrual embedding of U to U2}
\calU(\calL_R)/\calU^0(\calL_R)\to \calU_2(\calL_R)/\calU^0_2(\calL_R)
}
is injective; see e.g. \cite[Proposition 5.2]{blanchard2008properly}. Third, it was shown in \cite[Lemma 3]{paschke1981k} that the map \cref{eq:natrual embedding of U to U2} is indeed injective.
\end{proof}

With the establishment $K_1$-injectivity of \cref{eq:K1 injective group homomorphism}, our main result \cref{thm:R-local unitaries are path-connected} follows directly from the well-known $K_1$-group calculation of $\calL_R$, as we presented in \cref{prop:K1 group of L local algebra}. From this perspective, we have given an alternative proof of the $K_1$-injectivity for the $C^\ast$-algebra $\calL_R$. On the other hand, we believe a purely functional analytic proof of \cref{thm:R-local unitaries are path-connected} exists (see \cite{ChungShapiro2023} for more relevant discussion on motivation for this).


\begingroup
\let\itshape\upshape
\printbibliography

@ARTICLE{Bellissard_1994JMP....35.5373B,
   author = {{Bellissard}, J. and {van Elst}, A. and {Schulz-Baldes}, H.},
    title = "{The noncommutative geometry of the quantum Hall effect}",
  journal = {J. Math. Phys.},
   eprint = {cond-mat/9411052},
     year = 1994,
    month = oct,
   volume = 35,
    pages = {5373-5451},
      doi = {10.1063/1.530758},
   adsurl = {http://adsabs.harvard.edu/abs/1994JMP....35.5373B}
}

@article{AndruchowChiumientoLucero2015,
  title={Essentially commuting projections},
  author={Andruchow, Esteban and Chiumiento, E and y Lucero, ME Di Iorio},
  journal={Journal of Functional Analysis},
  volume={268},
  number={2},
  pages={336--362},
  year={2015},
  publisher={Elsevier}
}

@BOOK{Douglas1998,
	title     = "Banach algebra techniques in operator theory",
	author    = "Douglas, Ronald G",
	publisher = "Springer",
	series    = "Graduate texts in mathematics",
	edition   =  2,
	month     =  jun,
	year      =  1998,
	address   = "New York, NY",
	language  = "en"
}

@article{CareyHurstOBrien1982,
	title={Automorphisms of the canonical anticommutation relations and index theory},
	author={Carey, AL and Hurst, CA and O'Brien, DM},
	journal={Journal of Functional Analysis},
	volume={48},
	number={3},
	pages={360--393},
	year={1982},
	publisher={Elsevier}
}

@book{rordam2000introduction,
  title={An introduction to $K$-theory for $C^\ast$-algebras},
  author={R{\o}rdam, Mikael and Larsen, Flemming and Larsen, Flemming and Laustsen, N},
  number={49},
  year={2000},
  publisher={Cambridge University Press}
}

@article{ChungShapiro2023,
  title={Topological Classification of Insulators: I. Non-interacting Spectrally-Gapped One-Dimensional Systems},
  author={Chung, Jui-Hui and Shapiro, Jacob},
  journal={arXiv preprint arXiv:2306.00268},
  year={2023}
}

@book{conway2000course,
  title={A course in operator theory},
  author={Conway, John B},
  year={2000},
  publisher={American Mathematical Soc.}
}

@book{kadison_fundamentals_1997,
	series = {Graduate studies in mathematics},
	title = {Fundamentals of the theory of operator algebras. {Volume} {I}: {Elementary} {Theory}},
	volume = {15},
	publisher = {American Mathematical Society},
	author = {Kadison, Richard V. and Ringrose, John R.},
	year = {1997},
}

@article{chung2024topological,
  title={Topological Classification of Insulators: II. Quasi-Two-Dimensional Locality},
  author={Chung, Jui-Hui and Shapiro, Jacob},
  journal={arXiv preprint arXiv:2406.05385},
  year={2024}
}

@article{gilfeather1983commutants,
  title={Commutants modulo the compact operators of certain CSL algebras II},
  author={Gilfeather, Frank and Larson, David R},
  journal={Integral Equations and Operator Theory},
  volume={6},
  pages={345--356},
  year={1983},
  publisher={Springer}
}

@book{higson2000analytic,
  title={Analytic $K$-homology},
  author={Higson, Nigel and Roe, John},
  year={2000},
  publisher={OUP Oxford}
}

@misc{private_shanshan,
  author       = {Shanshan Hua},
  howpublished = {{Private Communication}},
  year         = {2024},
}

@article{blanchard2008properly,
  title={Properly infinite $ C (X) $-algebras and $ K_1 $-injectivity},
  author={Blanchard, Etienne and Rohde, Randi and R{\o}rdam, Mikael},
  journal={Journal of noncommutative geometry},
  volume={2},
  number={3},
  pages={263--282},
  year={2008}
}

@article{paschke1981k,
  title={$K$-theory for commutants in the Calkin algebra},
  author={Paschke, William},
  journal={Pacific Journal of Mathematics},
  volume={95},
  number={2},
  pages={427--434},
  year={1981},
  publisher={Mathematical Sciences Publishers}
}

@article{loreaux2020remarks,
  title={Remarks on essential codimension},
  author={Loreaux, Jireh and Ng, Ping W},
  journal={Integral equations and operator theory},
  volume={92},
  pages={1--35},
  year={2020},
  publisher={Springer}
}

@article{hua2024k,
  title={$K$-stability of $\mathcal{Z}$-stable $C^*$-algebras},
  author={Hua, Shanshan},
  journal={arXiv preprint arXiv:2406.11084},
  year={2024}
}
\endgroup
\end{document}